\documentclass[11pt]{article}
\usepackage{amsmath, amssymb, amsthm}
\usepackage{verbatim}
\usepackage{multicol}
\usepackage{enumerate}
\usepackage{comment}
\usepackage[none]{hyphenat}
\usepackage[unicode]{hyperref}
\hypersetup{
	colorlinks=true,
	linkcolor=blue,
	filecolor=magenta,      
	urlcolor=cyan,
	citecolor=blue
}
\usepackage{pgf}
\usepackage{tikz}
\usetikzlibrary{positioning,arrows,shapes,decorations.markings,decorations.pathreplacing,matrix,patterns}
\tikzstyle{vertex}=[circle,draw=black,fill=black,inner sep=0,minimum size=3pt,text=white,font=\footnotesize]
\usepackage{cleveref}

\date{}
\title{\vspace{-1.2cm} Ramsey properties of semilinear graphs}

\author{
Istv\'an Tomon\thanks{ETH Zurich, \emph{e-mail}: \textbf{istvan.tomon@math.ethz.ch}}
}

\theoremstyle{plain}
\newtheorem{theorem}{Theorem}[section]

\newtheorem{lemma}[theorem]{Lemma}

\Crefname{theorem}{Theorem}{Theorems}
\Crefname{definition}{Definition}{Definitions}
\Crefname{corollary}{Corollary}{Corollaries}
\Crefname{claim}{Claim}{Claims}
\Crefname{lemma}{Lemma}{Lemmas}
\Crefname{conjecture}{Conjecture}{Conjectures}
\Crefname{problem}{Problem}{Problems}
\Crefname{prop}{Proposition}{Propositions}

\theoremstyle{definition}

\DeclareMathOperator{\polylog}{polylog}

\newcommand{\m}{\mathbf}

\begin{document}

\maketitle
\sloppy

\begin{abstract}
A graph $G$ is \emph{semilinear of complexity $t$} if the vertices of $G$ are elements of $\mathbb{R}^{d}$ for some $d\in\mathbb{Z}^{+}$, and the edges of $G$ are defined by the sign patterns of $t$ linear functions $f_1,\dots,f_t:\mathbb{R}^{d}\times \mathbb{R}^{d}\rightarrow\mathbb{R}$. We show that semilinear graphs of constant complexity have very tame Ramsey properties. More precisely, we prove that if $G$ is a semilinear graph of complexity $t$ which contains no clique of size $s$ and no independent set of size $n$, then $G$ has at most $O_{s,t}(n)\cdot(\log n)^{O_t(1)}$ vertices. We also show that the logarithmic term cannot be omitted. In particular, this implies that if $G$ is a semilinear graph of constant complexity on $n$ vertices, and $G$ contains no clique of size $s$, then $G$ can be properly colored with $\polylog(n)$ colors. In the past 60 years, this coloring question was extensively studied for several special instances of semilinear graphs, e.g. shift graphs, intersection and disjointness graphs of certain geometric objects, and overlap graphs. Our main result provides a general upper bound on the chromatic number of all such, seemingly unrelated, graphs.

Furthermore, we consider the symmetric Ramsey problem for semilinear graphs as well. It is known that if there exists an intersection graph of $N$ boxes in $\mathbb{R}^{d}$ (such graphs are semilinear of complexity $2d$) that contains no clique or independent set of size $n$, then $N=O_d(n^2(\log n)^{d-1})$. That is, the exponent of $n$ does not grow with the dimension. We prove a result about the symmetric Ramsey properties of semilinear graphs, which puts this phenomenon in a more general context.
\end{abstract}

\section{Introduction}
By Ramsey's theorem, if $s$ and $n$ are positive integers, then there exists a smallest number $N=R(s,n)$ such that any graph on $N$ vertices contains either a clique of size $s$, or an independent set of size $n$. Estimating the numbers $R(s,n)$ is one of the central problems of graph theory, due to its wide applications in geometry, computer science, logic and number theory. In the \emph{asymmetric Ramsey problem}, we are interested in the behavior of $R(s,n)$ as a function of $n$, while we think of $s$ as a constant. By the celebrated result of Erd\H{o}s and Szekeres \cite{ESz35} from 1935, we have $R(s,n)\leq \binom{s+n-2}{s-1}<n^{s-1}$, and currently the best known bounds are $n^{\frac{s+1}{2}-o(1)}<R(s,n)<n^{s-1-o(1)}$, see \cite{BK10} and \cite{AKSz80}. In the \emph{symmetric Ramsey problem}, we are interested in the value of $R(n,n)$. The above mentioned result of Erd\H{o}s and Szekeres \cite{ESz35} implies that $R(n,n)< 4^n$, while the celebrated probabilistic argument of Erd\H{o}s \cite{E47} shows that $R(n,n)=\Omega(2^{n/2})$. These bounds have been only slightly improved since.

However, if we restrict our attention to graphs that arise from geometric settings, these bounds can be greatly improved. There is a huge literature studying the phenomenon that for certain geometrically defined graphs $G$, if $G$ contains no clique of size $s$ and no independent set of size $n$, then it has at most linear, or almost linear number of vertices in $n$. These include intersection and disjointness graphs of certain geometric objects, overlap graphs, and shift graphs. We provide a detailed discussion of such results in Section \ref{sect:asym}. In this paper, we study the Ramsey properties of a general family of graphs, called semilinear graphs, that collects many geometrically defined graphs for which the above phenomenon has been observed. Our main result provides an explanation for the tame Ramsey properties of these graph families.

Semilinear graphs and hypergraphs were recently introduced by Basit, Chernikov, Starchenko, Tao and Tran \cite{BCSTT20}, who studied Zarankiewicz's problem  (also known as extremal, or Tur\'an-type problems) in these families. Let us provide two definitions, which are (more or less) interchangeable. A function $f:\mathbb{R}^{n}\rightarrow \mathbb{R}$ is \emph{linear} if there exist $n+1$ real numbers $a_1,\dots,a_n,b$ such that $$f(\m{x})=b+\sum_{i=1}^{n}a_{i}\m{x}(i).$$ 
Given a positive integer $t$, say that a graph $G$ is \emph{semilinear of complexity $t$}, if $V(G)\subset \mathbb{R}^{d}$ for some positive integer $d$, and there exist $t$ linear functions $f_{1},\dots,f_t:\mathbb{R}^{d}\times\mathbb{R}^d\rightarrow \mathbb{R}$, and a Boolean function $\phi:\{\mbox{F},\mbox{T}\}^{3t}\rightarrow \{\mbox{F},\mbox{T}\}$ (where F and T denote false and true, respectively) such that for $\m{x},\m{y}\in V(G)$, $\{\m{x},\m{y}\}$ is an edge of $G$ if and only if 
\begin{equation}\label{equ:def1}
\phi(\{f_{i}(\m{x},\m{y})<0,f_{i}(\m{x},\m{y})\leq 0,f_{i}(\m{x},\m{y})=0\}_{i\in [t]})=\mbox{T}.
\end{equation}
In other words, whether $\{\m{x},\m{y}\}$ is an edge of $G$ is determined by the \emph{sign-pattern} of $(f_1(\m{x},\m{y}),\dots,f_t(\m{x},\m{y}))$. Also, say that $G$ is \emph{semilinear*} of complexity $(t,u)$ if there exist $t\cdot u$ linear functions $f_{i,j}:\mathbb{R}^{d}\times\mathbb{R}^d\rightarrow \mathbb{R}$ for $(i,j)\in [u]\times [t]$ such that for $\m{x},\m{y}\in V(G)$, $\{\m{x},\m{y}\}$ is an edge of $G$ if and only if \begin{equation}\label{equ:def2}\bigvee_{i\in [u]}\left(\bigwedge_{j\in [t]}\{f_{i,j}(\m{x},\m{y})<0\}\right)=\mbox{T}.\end{equation}
We assume that for every $\m{x},\m{y}\in V(G)$, the left hand side of (\ref{equ:def1}) and (\ref{equ:def2}) remains the same after switching $\m{x}$ and $\m{y}$, so the edges of $G$ are well defined. Later, we will show that these two definitions are equivalent up to the values of $t$ and $u$. We provide these two definitions, as in certain situations it is more convenient to work with one than the other.

\subsection{Asymmetric Ramsey properties of semilinear graphs}\label{sect:asym}

 In \cite{BCSTT20}, it was proved that if $G$ is a semilinear graph of complexity $t$ on $n$ vertices, and $G$ does not contain the complete bipartite graph $K_{k,k}$, then $G$ has at most $O_{t,k}(n)\cdot (\log n)^{O_t(1)}$ edges. That is, semilinear graphs have much tamer extremal properties than general graphs. We show an analogue of this result for Ramsey properties. Let $R_{t}(s,n)$ denote the smallest $N$ such that every semilinear graph of complexity $t$ on $N$ vertices contains either a clique of size $s$, or an independent set of size $n$. One of the main results of our paper is the following.

\begin{theorem}\label{thm:mainupper0}
Let $t$ be a positive integer, then there exist $\alpha,\beta,\gamma>0$ such that $$R_t(s,n)\leq \alpha s^{\beta} n\cdot (\log n)^{\gamma}.$$
\end{theorem}

We also show that the polylogarithmic term cannot be omitted.

\begin{theorem}\label{thm:mainlower}
There exist $t$ and $c>0$ such that $$R_t(3,n)\geq cn\frac{\log n}{\log\log n.}$$
\end{theorem}

The \emph{girth} of a graph is the length of its shortest cycle. What happens if instead of forbidding a triangle or a clique, we forbid small girth? A celebrated probabilistic argument of Erd\H{o}s \cite{Er59} shows that for every $g$, there exist graphs with girth $g$ and independence number $n$ having superlinear number of vertices with respect to $n$. We prove the following extension of Theorem \ref{thm:mainlower}, which shows that semilinear graphs of constant complexity, girth $g$ and independence number $n$ can also have superlinear number of vertices with respect to $n$.

\begin{theorem}\label{thm:lower2}
There exists a positive integer $t$ such that for every positive integer $g$ the following holds. For every $n$, there exists a semilinear graph $G$ of complexity $t$ such that the girth of $G$ is at least $g$, $G$ contains no independent set of size $n$, and $G$ has at least $cn\log\log n$ vertices, where $c>0$ depends only on $g$ and $t$.
\end{theorem}

In order to prove Theorem \ref{thm:lower2}, we show that there exist $n$ points and $n$ open rectangles in the plane such that their incidence graph $G$ has girth at least $g$, and $G$ has $\Omega_g(n\log\log n)$ edges. This extends (with slightly weaker bounds) a construction of \cite{BCSTT20} which dealt with the case $g=6$, and improves a result of \cite{D20}. See Section \ref{sect:construction} for more details.

\medskip

 Theorem \ref{thm:mainupper0}  is an immediate consequence of the following theorem about the coloring properties of semilinear graphs. If $G$ is a graph, $\chi(G)$ denotes the chromatic number of $G$.

\begin{theorem}\label{thm:mainthm}
Let $t$ be a positive integer, then there exist $\alpha,\beta,\gamma>0$ such that the following holds. Let $s,n$ be positive integers, and let $G$ be a semilinear graph  of complexity $t$ on $n$ vertices, which contains no clique of size $s$. Then $$\chi(G)\leq \alpha s^{\beta}\cdot(\log n)^{\gamma}.$$
\end{theorem}

Indeed, in a proper coloring, every colorclass is an independent set, so Theorem \ref{thm:mainthm} implies Theorem~\ref{thm:mainupper0}. In the past 60 years, many different instances of Theorem \ref{thm:mainthm} have been considered.  Given a family $\mathcal{F}$ of geometric objects, the \emph{intersection graph of $\mathcal{F}$} is the graph, whose vertices are the elements of $\mathcal{F}$, and two vertices are joined by an edge if the corresponding sets have a nonempty intersection. Also, the \emph{disjointness graph of $\mathcal{F}$} is the complement of the intersection graph. One of the first results in this long line of research is the celebrated result of Asplund and Gr\"unbaum \cite{AG60}, which states that if $G$ is the intersection graph of rectangles in the plane, and $G$ contains no clique of size $s$, then $\chi(G)=O(s^{2})$.  This was only recently improved to $O(s\log s)$ by Chalermsook and Walczak \cite{CW20}. Note that the intersection graph of boxes in $\mathbb{R}^{d}$ is semilinear of complexity $2d$: an open box $R$ is defined by a $2d$-tuple $(a_1,\dots,a_d,b_1,\dots,b_d)$ such that $R=\{(x_1,\dots,x_d)\in \mathbb{R}^{d}:\forall i\in [d], a_i< x_i<b_i\}$, and the boxes defined by $(a_1,\dots,a_d,b_1,\dots,b_d)$ and $(a_1',\dots,a_d',b_1',\dots,b_d')$ intersect if and only if $a_i'<b_i$ and $a_i<b_i'$ hold for all $i\in [d]$. In general, if $G$ is the intersection graph of $n$ boxes in $\mathbb{R}^d$, and $G$ contains no clique of size $s$, then $\chi(G)=O(s(\log n)^{d-1})$, and a construction of Burling \cite{B65} shows that in case $d,s\geq 3$, there exists such $G$ with $\chi(G)=\Omega(\log\log n)$. On the other hand, K\'arolyi \cite{K91} proved that if $G$ is the disjointness graph of boxes in $\mathbb{R}^d$, and $G$ contains no clique of size $s$, then $\chi(G)=O(s(\log s)^{d-1})$. Recently, Davies \cite{D20} proved that there exist intersection graphs of boxes in $\mathbb{R}^{3}$ with $n$ vertices, girth $g$, and chromatic number $\alpha_g(n)$, where $\alpha_g(n)$ is a function tending to infinity extremely slowly (even compared to $\log\log n$).
Another interesting family of semilinear graphs is the intersection graph of L-shapes. An \emph{L-shape} is a vertical and a horizontal segment joined at the lower endpoint of the formal, and left endpoint of the latter. Walczak \cite{W19} showed that a triangle-free intersection graph of $n$ $L$-shapes has chromatic number at most $O(\log\log n)$, and this bound is the best possible.

One of the most elegant constructions of triangle-free graphs with arbitrarily large chromatic number is due to Erd\H{o}s and Hajnal \cite{EH64}. For positive integers $m\geq k\geq 2$, the \emph{shift graph $S(m,k)$} is the graph, whose vertices are the $k$-tuples of integers $(x_1,\dots,x_k)$ satisfying $1\leq x_1<\dots<x_k\leq m$, and $(x_1,\dots,x_k)$ and $(y_1,\dots,y_k)$ are joined by an edge if $x_{i+1}=y_i$ for $i=1,\dots,k-1$. Clearly, the shift graph $S(m,k)$ is semilinear of complexity $k-1$, it has $n=\binom{m}{k}$ vertices, and contains no triangles. Also, its chromatic number is $(1+o(1)) \log_2\dots\log_2 m$, where the $\log_2$ is iterated $(k-1)$-times. In particular, $S(m,2)$ has chromatic number $\lfloor \log_2 m\rfloor$, but it is not a good example for Theorem \ref{thm:mainlower}, as it contains linear sized independent sets. 

 Given a family $\mathcal{F}$ of geometric objects, its \emph{overlap graph} is the graph whose vertices are the elements of $\mathcal{F}$, and two vertices are joined by an edge if the corresponding sets have a nonempty intersection, and none of them contains the other. Clearly, overlap graphs of boxes are semilinear of constant complexity. It was proved by Gy\'arf\'as \cite{Gy85} that if $G$ is  an overlap graph of intervals (which is the same as an intersection graph of chords of a cycle), and $G$ contains no clique of size $s$, then $\chi(G)=2^{O(s)}$. This was only recently improved to $\chi(G)=O(s^2)$ by Davies and McCarty \cite{DM19}. If $G$ is the triangle-free overlap graph of $n$ rectangles, then $\chi(G)=O(\log\log n)$ \cite{KPW15}, and this bound is the best possible.
 
 \subsection{Symmetric Ramsey properties of semilinear graphs}
 
 Semilinear graphs of constant complexity are special instances of semialgebraic graphs of constant complexity. A graph $G$ is \emph{semialgebraic} of complexity $t$, if $V(G)\subset \mathbb{R}^{d}$ for some $d\leq t$, and the edges of $G$ are defined by the sign-patterns of $t$ polynomials $f_1,\dots,f_t:\mathbb{R}^{d}\times\mathbb{R}^d\rightarrow \mathbb{R}$ of degree at most $t$. We remark that the analogue of Theorem \ref{thm:mainupper0} does not hold for semialgebraic graphs. See, for example, the recent construction of Suk and the author of this paper \cite{ST21}, which shows that for every $n$, there exists a semialgebraic graph of constant complexity on $\Omega(n^{4/3})$ vertices which contains no triangle and no independent set of size~$n$. 
 
  The (symmetric) Ramsey properties of semi-algebraic graphs were first studied by
Alon, Pach, Pinchasi, Radoi\v ci\'c, and Sharir \cite{APPRS}. They proved that if $G$ is a semialgebraic graph of complexity $t$ which contains no clique or independent set of size $n$, then $G$ has at most $n^{O_t(1)}$ vertices. It might be tempting to conjecture that semilinear graphs have even tamer symmetric Ramsey properties. Indeed, it follows from the aforementioned theorem of K\'arolyi \cite{K91}, and was also proved in \cite{LMPT94}, that if $G$ is an intersection graph of boxes in $\mathbb{R}^d$ which contains no clique or independent set of size $n$, then $G$ has at most $O_d(n^2(\log n)^{d-1})$ vertices. That is, the exponent of $n$ does not grown with the complexity. However, this is too much to ask for in general. One of the most well know explicit constructions of graphs with good Ramsey properties is due to Frankl and Wilson \cite{FW81}. Let $p$ be a prime, and let $G$ be the graph, whose vertices are the $p^2-1$ element subsets of $[m]$, and two such sets, $A$ and $B$ are joined by an edge if $|A\cap B|\equiv -1 \pmod p$. Note that this graph is semilinear of complexity $(p^2-1)^2$. Indeed, we can represent each set $A$ as a vector $\mathbf{x}_{A}\in \mathbb{R}^{p^2-1}$ by listing its elements, then the edges of $G$ are defined by the zero-patterns of the $(p^2-1)^2$  polynomials $f_{i,j}(\mathbf{x},\mathbf{y})=\mathbf{x}(i)-\mathbf{y}(j)$ for $(i,j)\in [p^2-1]^2$. The graph $G$ has $N=\binom{m}{p^2-1}$ vertices, and by the celebrated Frankl-Wilson theorem on restricted intersections, $G$ contains no clique or independent set of size larger than $n=\binom{m}{p-1}$. As $N=\Omega_p(n^{p+1})$, we have $R_t(n,n)=\Omega_p(n^{p+1})$ for $t=(p^2-1)^2$.

However, we can still give an explanation why intersection graphs of boxes have very tame symmetric Ramsey properties. Let $R_{t,u}(n)$ denote the smallest $N$ such that every semilinear* graph of complexity $(t,u)$ on $N$ vertices contains either a clique or an independent set of size $n$. Note that the intersection graph of boxes in dimension $d$ is semilinear* of complexity $(2d,1)$. In general, we prove that the order of $R_{t,u}(n)$ depends only on $u$.

\begin{theorem}\label{thm:sym}
Let $t,u$ be positive integers. Then there exist $\alpha=\alpha(t,u)$ and $\beta=\beta(u)$ such that  $$R_{t,u}(n)\leq \alpha n^{\beta}.$$
\end{theorem}
 
 \medskip
 
 Our paper is organized as follows. In the next section, we present our notation and prepare the proofs of our main results. Then, in Section \ref{sect:color}, we prove Theorem \ref{thm:mainthm}, and in Section \ref{sect:sym}, we prove Theorem \ref{thm:sym}. In Section \ref{sect:construction}, we present our constructions, that is, we prove Theorems \ref{thm:mainlower} and \ref{thm:lower2}. Then, we conclude our paper with some remarks and open problems. We omit the use of floors and ceilings whenever they are not crucial.

\section{Preliminaries}\label{sect:prelim}

Let us first show that the two notions of semilinearity can be exchanged up to the value of $t$ and $u$. Clearly, if $G$ is semilinear* of complexity $(t,u)$, then it is semilinear of complexity $t\cdot u$. Let us show the other direction as well.

\begin{lemma}\label{lemma:def}
 For every $t$, there exist $t'$ and $u$ such that if $G$ is semilinear of complexity $t$, then $G$ is semilinear* of complexity $(t',u)$.
\end{lemma}

\begin{proof}
Let $f_1,\dots,f_t:\mathbb{R}^{d}\times\mathbb{R}^{d}\rightarrow \mathbb{R}$ be linear functions and $\phi:\{\mbox{F},\mbox{T}\}^{3t}\rightarrow \{\mbox{F},\mbox{T}\}$ be a Boolean formula defining $G$ with respect to (\ref{equ:def1}). Then there exist a positive integer $u$ and $\epsilon_{i,j}\in \{-1,0,1\}$ for $(i,j)\in [u]\times [3t]$ such that
$$\phi(b_1,\dots,b_{3t})=\bigvee_{i\in [u]}\left(\bigwedge_{j\in [3t]}b_{j}^{\epsilon_{i,j}}\right),$$
where $b^{-1}=\neg b$, $b^{0}=\mbox{T}$ and $b^{1}=b$ for $b\in\{\mbox{F},\mbox{T}\}$.
Note that if $b=\{f_i(\m{x},\m{y})<0\}$, then $b^{-1}=\{-f_i(\m{x},\m{y})\leq 0\}$. Also, if $b=\{f_i(\m{x},\m{y})=0\}$, then $b=\{f_i(\m{x},\m{y})\leq 0\}\wedge \{-f_i(\m{x},\m{y})\leq 0\}$, and  $b^{-1}=\{f_i(\m{x},\m{y})< 0\}\vee \{-f_i(\m{x},\m{y})< 0\}$. Finally, using that $V(G)$ is finite, there exists a sufficiently small $\epsilon>0$ such that for every $\m{x},\m{y}\in V(G)$, we have  $\{f_i(\m{x},\m{y})\leq 0\}=\{f_i(\m{x},\m{y})-\epsilon< 0\}.$
Therefore, we can find $t'$ depending only $t$, and $t'u$ functions $f_{i,j}\in \{f_\ell,-f_\ell,f_\ell-\epsilon,-f_\ell-\epsilon:\ell\in [t]\}$ for $(i,j)\in [u]\times [t']$ which define $G$ with respect to (\ref{equ:def2}).
\end{proof}

We say that a graph $G$ is a \emph{comparability graph} if there exists a partial ordering $\prec$ on $V(G)$ such that $\{x,y\}\in E(G)$ if and only if $x\prec y$ or $y\prec x$. Also, let $\omega(G)$ denote the clique number of $G$. We will use the following property of comparability graphs, which is the consequence of the dual of Dilworth's theorem (also known as Mirsky's theorem).

\begin{lemma}\label{lemma:dilworth}
Let $G$ be a comparability graph. Then $\chi(G)=\omega(G)$.
\end{lemma}

If $G_1,\dots,G_t$ are graphs on the same vertex set, then $G_1\cap\dots\cap G_t$ and $G_1\cup\dots\cup G_t$ denote the graphs on the vertex set $V(G_1)$ whose edge sets are $E(G_1)\cap\dots\cap E(G_t)$ and $E(G_1)\cup\dots\cup E(G_t)$, respectively. 

\begin{lemma}\label{lemma:union}
Let $G_1,\dots,G_t$ be graphs on the same vertex set. Then
$$\chi(G_1\cup\dots\cup G_t)\leq \chi(G_1)\dots\chi(G_t).$$
\end{lemma}

\begin{proof}
Let $c_{i}:V(G_i)\rightarrow\mathbb{N}$ be a proper coloring of $G_i$ with $\chi(G_i)$ colors. Then the coloring $c:V(G_1)\rightarrow \mathbb{N}^{t}$ defined as $c(v)=(c_1(v),\dots,c_t(v))$ for $v\in V(G_1)$ is a proper coloring of $G_1\cup\dots\cup G_t$ using at most $\chi(G_1)\dots\chi(G_t)$ colors.
\end{proof}

\section{Coloring semilinear graphs}\label{sect:color}

In this section, we prove Theorem \ref{thm:mainthm}. If $\m{x},\m{y}\in \mathbb{R}^{t}$, then write $\m{x}\prec \m{y}$ if $\m{x}(i)<\m{y}(i)$ for every $i\in [t]$. Then $\prec$ is the coordinate-wise partial ordering on $\mathbb{R}^{t}$.  Say that $G$ is a \emph{quasi-comparability graph of complexity $t$} if $V(G)\subset\mathbb{R}^{t}\times\mathbb{R}^t$, and $(\m{x},\m{y})$ and $(\m{x}',\m{y}')$ are joined by an edge if and only if $\m{x}\prec \m{y}'$ or $\m{x}'\prec \m{y}$. Note that in case $\m{x}=\m{y}$ for every vertex $(\m{x},\m{y})$ of $G$, then $G$ is the comparability graph of a poset of (Duschnik-Miller) dimension $t$. In this section, we prove the following theorem, which then easily implies Theorem \ref{thm:mainthm}.

\begin{theorem}\label{thm:quasi}
For every positive integer $t$, there exist $\alpha,\beta,\gamma>0$ such that the following holds. Let $s,n$ be positive integers, and let $G$ be a quasi-comparability graph of complexity $t$ on $n$ vertices, which contains no clique of size $s$. Then $$\chi(G)\leq \alpha s^\beta (\log n)^{\gamma}.$$
\end{theorem}

First, let us show that Theorem \ref{thm:quasi} indeed implies Theorem \ref{thm:mainthm}.

\begin{proof}[Proof of Theorem \ref{thm:mainthm}]
By Lemma \ref{lemma:def}, there exist $t'$ and $u$ depending only on $t$ such that $G$ is semilinear* of complexity $(t',u)$. With slight abuse of notation, write $t$ instead of $t'$. Let $\alpha',\beta',\gamma'$ be the constants given by Theorem \ref{thm:quasi} with respect to $t$, instead of $\alpha,\beta,\gamma$, respectively. For $(i,j)\in [u]\times [t]$, let $f_{i,j}:\mathbb{R}^d\times\mathbb{R}^d\rightarrow \mathbb{R}$ be linear functions that define $G$ according to (\ref{equ:def2}). That is, for $\m{x},\m{y}\in V(G)$, $\{\m{x},\m{y}\}$ is an edge if and only if 
$$\bigvee_{i\in [u]}\left(\bigwedge_{j\in [t]}\{f_{i,j}(\m{x},\m{y})<0\}\right)=\mbox{T}.$$ 
Note that as $f_{i,j}$ is linear, we can write $f_{i,j}(\m{x},\m{y})=g_{i,j}(\m{x})+h_{i,j}(\m{y})$ with suitable functions $g_{i,j},h_{i,j}:\mathbb{R}^{d}\rightarrow\mathbb{R}$. For each $\m{x}\in V(G)$, define $\phi_j(\m{x}),\rho_j(\m{x})\in\mathbb{R}^{t}$ such that $\phi_i(\m{x})(j)=g_{i,j}(\m{x})$ and $\rho_i(\m{x})(j)=-h_{i,j}(\m{x})$ for $j\in [t]$. Then $$\bigwedge_{j\in [t]}\{f_{i,j}(\m{x},\m{y})<0\}\equiv \{\phi_i(\m{x})\prec \rho_i(\m{y})\}.$$ Therefore, $\{\m{x},\m{y}\}$ is an edge of $G$ if and only if $\phi_i(\m{x})\prec \rho_i(\m{y})$ or $\phi_i(\m{y})\prec \rho_i(\m{x})$ for some $i\in [t]$. Let $G_i$ be the graph on $V(G)$ in which $\{\m{x},\m{y}\}$ is an edge if and only if $\phi_i(\m{x})\prec \rho_i(\m{y})$ or $\phi_i(\m{y})\prec \rho_i(\m{x})$. Then $G_i$ is isomorphic to a quasi-comparability graph of complexity $t$, and $G=\bigcup_{i\in [t]}G_i$. Therefore, $G_i$ contains no clique of size $s$, which implies $\chi(G_{i})\leq \alpha' s^{\beta'}(\log n)^{\gamma'}$ by Theorem \ref{thm:quasi}. Finally, by applying Lemma \ref{lemma:union}, we get
$$\chi(G)\leq \prod_{i\in [u]}\chi(G_i)\leq (\alpha')^u s^{u\beta'}(\log n)^{u\gamma'}.$$
Hence, the choices $\alpha:=(\alpha')^u$, $\beta:=u\beta'$ and $\gamma:=u\gamma'$ suffice.
\end{proof}

In the rest of this section, we prove Theorem \ref{thm:quasi}. We prepare its proof with the following lemma, which tells us that the intersection of a comparability graph and an intersection graph of boxes can be colored with a few colors.

\begin{lemma}\label{lemma:intersection}
Let $d$ be a positive integer, then there exists $c>0$ such that the following holds. Let $G$ be the intersection graph of $n$ open boxes in $\mathbb{R}^{d}$, and let $H$ be a comparability graph on $V(G)$. If $G\cap H$ contains no clique of size $s$, then  $$\chi(G)\leq cs(\log n)^{d}.$$
\end{lemma}

\begin{proof}
Let $h_{d,s}(n)$ be the minimal positive integer $\chi$ such that for every $G'$ that is the intersection graph of $n$ boxes in $\mathbb{R}^{d}$, and for every $H'$ that is a comparability graph on $V(G')$, if $G'\cap H'$ contains no clique of size $s$, then its chromatic number is at most $\chi$. We prove by induction on $n$ and $d$ such that $h_{d,s}(n)\leq c(d)s(\log n)^{d},$ where $c(d)>0$ depends only on $d$. Clearly, $h_{d,s}(1)=1$.

Suppose that $G\cap H$ contains no clique of size $s$. Let $\mathcal{B}$ be the set of boxes in $\mathbb{R}^{d}$, whose intersection graph is $G$. Given $h\in\mathbb{R}$, let $\mathcal{B}(h)^{-}$ be the set of boxes in $\mathcal{B}$ that are contained in the half-space $\{x(d)<h:x\in \mathbb{R}^d\}$, and let $\mathcal{B}(h)^{+}$ be the set of boxes in $\mathcal{B}$ that are contained in the half-space $\{x(d)>h:x\in \mathbb{R}^d\}$. Finally, let $\mathcal{B}(h)$ be the set of boxes that intersect the hyperplane $\mathcal{H}(h)=\{x(d)=h:x\in \mathbb{R}^d\}$. It is easy to see that there exists $h$ such that $|\mathcal{B}(h)^{-}|,|\mathcal{B}(h)^{+}|\leq n/2$. Let $\mathcal{B}'=\{B\cap \mathcal{H}(h):B\in \mathcal{B}\}$, then $\mathcal{B}'$ is a set of $(d-1)$-dimensional boxes, whose intersection graph is isomorphic to the  intersection graph of $\mathcal{B}(h)$.

Let $G^{-}$, $G^{+}$ and $G^{0}$ be the subgraphs of $G$ induced by the sets $\mathcal{B}(h)^{-},\mathcal{B}(h)^{+}$ and $\mathcal{B}(h)$, respectively, and let $H^{-},H^{+},H^{0}$ be the subgraphs of $H$ induced by $\mathcal{B}(h)^{-},\mathcal{B}(h)^{+}$ and $\mathcal{B}'$, respectively. Note that if $B\in \mathcal{B}(h)^{-}$ and $B'\in \mathcal{B}(h)^{+}$, then $B$ and $B'$ are disjoint, so there is no edge between $B$ and $B'$ in $G\cap H$. Therefore, the subgraph of $G\cap H$ induced on $V(G^{-})\cup V(G^{+})$ can be properly colored by $$\max\{\chi(G^{-}\cap H^{-}),\chi(G^{+}\cap H^{+})\}\leq h_{d,s}\left(\frac{n}{2}\right)$$ colors.

Also, if  $d=1$, then any two boxes in $\mathcal{B}(h)$ (which are actually intervals) have a nonempty intersection. Hence, $G^{0}\cap H^{0}$ is equal to $H^{0}$ on the vertex set $\mathcal{B}(h)$. As $H^{0}$ is a comparability graph which contains no clique of size $s$, we have $\chi(G^{0}\cap H^{0})<s$. Also, if $d>1$, then  $\chi(G^{0}\cap H^{0})<h_{d-1,s}(n)$. Extend the definition of $h_{d,s}(n)$ for $d=0$ as $h_{0,s}(n)=s$, and let $c(0)=1$. Then for $d\geq 1$, we can write
$$\chi(G\cap H)\leq h_{d,s}\left(\frac{n}{2}\right)+h_{d-1,s}(n)\leq c(d)s\left(\log \frac{n}{2}\right)^{d}+c(d-1)s(\log n)^{d-1}<c(d)s(\log n)^{d},$$
where the last inequality holds if $c(d)$ is sufficiently large compared to $c(d-1)$.
\end{proof}

Now we are ready to prove Theorem \ref{thm:quasi}.

\begin{proof}[Proof of Theorem \ref{thm:quasi}]
We can assume that for every $(\m{x},\m{y})\in V(G)$ we have $\m{x}(i)\neq \m{y}(i)$. Indeed, we can decrease each coordinate of $\m{y}$ by some small amount without changing the quasi-comparability graph generated by the points. For $\epsilon\in \{-,+\}^{t}$, say that a point $(\m{x},\m{y})$ has type $\epsilon$, if the sign of $\m{y}(i)-\m{x}(i)$ is $\epsilon(i)$ for $i\in [t]$, and let $V_{\epsilon}$ be the set of elements of $V$ of type $\epsilon$.

Fix some $\epsilon\in \{-,+\}^{t}$. Let $P\subset [t]$ be the set of indices $i$ such that $\epsilon(i)=+$. For $Q\subset P$, define the graphs $G_{Q}$ and $H_Q$ on $V_{\epsilon}$ as follows. If $(\m{x},\m{y}),(\m{x}',\m{y}')\in V_{\epsilon}$, then join $(\m{x},\m{y})$ and $(\m{x}',\m{y}')$ by an edge in $G_Q$ if for every $i\in P\setminus Q$, the open intervals $(\m{x}(i),\m{y}(i))$ and $(\m{x}'(i),\m{y}'(i))$ have a nonempty intersection. Then $G_Q$ is the intersection graph of $(|P|-|Q|)$-dimensional open boxes. Also, write $(\m{x},\m{y})\prec_{Q} (\m{x}',\m{y}')$  if and only if 
\begin{itemize}
    \item for $i\in [t]\setminus P$, $\m{x}(i)<\m{y}'(i)$ (which then implies $\m{y}(i)< \m{x}(i)<\m{y}'(i)< \m{x}'(i)$),
    \item for $i\in Q$, $\m{y}(i)<\m{x}'(i)$ (which then implies $\m{x}(i)< \m{y}(i)<\m{x}'(i)< \m{y}'(i)$). 
\end{itemize}
Then $\prec_Q$ is a partial order. Let $H_{Q}$ be the comparability graph of $\prec_{Q}$.

We claim that $G[V_{\epsilon}]=\bigcup_{Q\subset P}(G_{Q}\cap H_Q)$. Indeed, it is easy to check that if $(\m{x},\m{y}),(\m{x}',\m{y}')\in V_\epsilon$, then  $\m{x}\prec \m{y}'$ if and only if there exists a unique $Q\subset P$ such that $(\m{x},\m{y})\prec_Q (\m{x}',\m{y}')$ and $\{(\m{x},\m{y}),(\m{x}',\m{y}')\}\in E(G_{Q})$. But then, for every $Q\subset P$, $G_Q\cap H_Q$ has no clique of size $s$. Applying Lemma \ref{lemma:intersection}, we get  that there exists some constant $c=c(t)>0$ such that $$\chi(G_Q\cap H_Q)\leq cs(\log n)^{|P|-|Q|}<cs(\log n)^t.$$
By Lemma \ref{lemma:union}, we get
$$\chi(G[V_{\epsilon}])\leq \prod_{Q\subset P}\chi(G_Q\cap H_Q)\leq c^{2^t}s^{2^t}(\log n)^{t2^t}.$$
This finishes the proof as
$$\chi(G)\leq \sum_{\epsilon\in\{-,+\}^{t}}\chi(G[V_{\epsilon}])\leq 2^t c^{2^t} s^{2^t}(\log n)^{t2^t},$$
so the choices $\alpha:=2^t c^{2^t}$, $\beta:=2^{t}$ and $\gamma:=t2^t$ suffice.
\end{proof}

\section{Symmetric Ramsey properties of semilinear graphs}\label{sect:sym}

In this section, we prove Theorem \ref{thm:sym}. The heart of the proof is the following theorem about quasi-comparability graphs, which might be of independent interest.

\begin{theorem}\label{thm:quasisym}
 There exists $c>0$ such that the following holds. Let $t$ be a positive integer, and let $G$ be a quasi-comparability graph of complexity $t$ on $n$ vertices. Then $G$ contains either a clique or an independent set of size at least $\alpha n^{c},$ where $\alpha=\alpha(t)>0$.
\end{theorem}

Let us prepare the proof of this theorem. The \emph{Boolean lattice $2^{[t]}$} is the family of all subsets of $[t]$ ordered by inclusion.  For $i\in [t]$, let $F^{-}_t(i)=\{A\in 2^{[t]}:i\not\in A\}$ and $F^{+}_t(i)=\{A\in 2^{[t]}:i\in A\}$. Let $\omega:2^{[t]}\rightarrow \mathbb{R}_{\geq 0}$ be a weight function. If $\mathcal{A}\subset 2^{[t]}$, then the weight of $\mathcal{A}$ is $\omega(\mathcal{A})=\sum_{A\in\mathcal{A}}\omega(A)$. Say that $\omega$ is \emph{balanced}, if $\omega(F_{t}^{-}(i))\leq 1/2$ and $\omega(F_{t}^{+}(i))\leq 1/2$ for every $i\in [t]$. The following lemma is closely related to Theorem 3. in \cite{T20}, where it was used in a Ramsey type result concerning intersection graphs of curves.

\begin{lemma}\label{lemma:EH}
There exists $c>0$ such that the following holds. Let $t\geq 2$ be a positive integer and let $\omega:2^{[t]}\rightarrow \mathbb{R}_{\geq 0}$ be a balanced weight function such that $\omega(2^{[t]})\geq 9/10$. Then either
\begin{description}
    \item[(i)] $\omega(\emptyset)\geq 1/10$ and $\omega([t])\geq 1/10$, or
    \item[(ii)] there exist a positive integer $s$ and $s$ disjoint families $\mathcal{A}_1,\dots,\mathcal{A}_s\subset 2^{[t]}$ such that $$\sum_{i=1}^{s}\omega(\mathcal{A}_i)^{c}\geq 1,$$
    and for every $1\leq i<j\leq t$, every $A\in\mathcal{A}_i$ is incomparable to every $B\in \mathcal{A}_j$.
\end{description}
\end{lemma}

\begin{proof}
We show that $c=1/10$ suffices. We prove the following slightly stronger result by induction on $t$. Either 
\begin{description}
     \item[(i)'] $\omega(\emptyset),\omega([t])\geq 1/5-\sum_{i=3}^{t}(2/i)^{10}$
\end{description}
or (ii) holds. Note that $\sum_{i=3}^{\infty}(2/i)^{10}<1/10$, so this indeed implies our lemma.

Consider the base case $t=2$, and suppose that (i)' does not hold. Then either $\omega(\emptyset)<1/5$, or $\omega(\{1,2\})<1/5$. Without loss of generality, suppose that $\omega(\emptyset)<1/5$.  We have $$\frac{9}{10}\leq \omega(2^{[2]})=\omega(\emptyset)+\omega(\{1\})+\omega(\{2\})+\omega(\{1,2\}).$$
Also, as $\omega$ is balanced, we have $\omega(\{1\})+\omega(\{1,2\})\leq 1/2$ and $\omega(\{2\})+\omega(\{1,2\})\leq 1/2$, which implies $\omega(\{1\})+\omega(\emptyset)\geq 4/10$ and $\omega(\{2\})+\omega(\emptyset)\geq 4/10$. But then $\omega(\{1\}),\omega(\{2\})\geq 1/5$, so $s=2$ and the families $\mathcal{A}_1=\{\{1\}\}$ and $\mathcal{A}_2=\{\{2\}\}$ satisfy (ii), noting that $(1/5)^{1/10}+(1/5)^{1/10}>1$.

Now let $t\geq 3$ and suppose that (ii) does not hold. Consider the weights of the single element sets $\{i\}$ for $i\in [t]$. Suppose that at least $t/2$ of these sets have weight at least $(2/t)^{10}$. Then taking $s=\lceil t/2\rceil$ and $\mathcal{A}_1,\dots,\mathcal{A}_{s}$ to be families containing exactly one of such sets, we get $\sum_{i=1}^{s}\omega(\mathcal{A}_{i})^{1/10}\geq s\cdot (2/t)\geq 1$. This contradicts that (ii) does not hold, therefore, we can assume that more than $t/2$ of the weights $\omega(\{1\}),\dots,\omega(\{t\})$ are at most $(2/t)^{10}$.
Now consider the weights of the $(t-1)$-element sets. By taking complements and repeating the same argument as before, we can also assume that more than $t/2$ of the weights  $\omega([t]\setminus\{1\}),\dots,\omega([t]\setminus\{t\})$ are at most $(2/t)^{10}$. But then there exists $i\in [t]$ such that $\omega(\{i\}),\omega([t]\setminus\{i\})\leq (2/t)^{10}$. Without loss of generality, we can assume that $i=t$.

Now define the weight function $\omega':2^{[t-1]}\rightarrow\mathbb{R}_{\geq 0}$ as follows. For $A\subset [t-1]$, let $\omega'(A)=\omega(A)+\omega(A\cup\{t\})$. The following properties of $\omega'$ are easy to check: $\omega'(2^{[t-1]})=\omega(2^{[t]})$, $\omega'$ is balanced, and $\omega'$ does not satisfy (ii). The latter is true as if $s$ and $\mathcal{A}'_1,\dots,\mathcal{A}'_s\subset 2^{[t-1]}$ satisfy (ii) with respect to $\omega'$, then the $s$ families $\mathcal{A}_{i}=\{A,A\cup\{t\}:A\in\mathcal{A}_i'\}$ satisfy (ii) with respect to $\omega$. Hence, by our induction hypothesis, we have $\omega'(\emptyset),\omega'([t-1])\geq 1/10-\sum_{i=3}^{t-1}(2/i)^{10}$. But note that $\omega(\emptyset)=\omega'(\emptyset)-\omega(\{t\})\geq \omega'(\emptyset)-(2/t)^{10}$ and $\omega([t])=\omega'([t-1])-\omega([t]\setminus\{t\})\geq \omega'(\emptyset)-(2/t)^{10}$, so $\omega$ satisfies (i). This finishes the proof.
\end{proof}

Now we are ready to prove Theorem \ref{thm:quasisym}. The family of \emph{cographs} is the smallest family of graphs which contains the single vertex graph, and is closed under complementation and taking disjoint unions. It is well known that cographs are \emph{perfect}, that is, their chromatic number is equal to their clique number. In particular, any cograph on $n$ vertices contains either a clique or an independent set of size at least $n^{1/2}$. We will show that quasi-comparability graphs contain large cographs.

\begin{proof}[Proof of Theorem \ref{thm:quasisym}]

Similarly as in the proof of Theorem \ref{thm:quasi}, we can assume that for every $(\m{x},\m{y})\in V(G)$ we have $\m{x}(i)\neq \m{y}(i)$. Also, for $\epsilon\in \{-,+\}^{t}$, say that a vertex $(\m{x},\m{y})$ has type $\epsilon$, if the sign of $\m{y}(i)-\m{x}(i)$ is $\epsilon(i)$ for $i\in [t]$. Let $c'$ be the constant $c$ given by Lemma \ref{lemma:EH}, and let $\gamma=\min\{c',1/3\}$. Also, set $\beta(1)=1$, and define $\beta(t)=\frac{\alpha(t-1)}{(10t)^c}$. We prove that if $G$ is a quasi-comparability graph of complexity $t$ on $n$ vertices such that every vertex in $G$ has the same type, then $G$ contains a cograph of size at least $\beta(t) n^{\gamma}$. We proceed by induction on $t$, and while $t$ is fixed, we proceed by induction on $n$. First, consider the base case $t=1$. If every vertex of $G$ has type +, then $G$ is the complete graph, and if every vertex  has type -, then $G$ is the intersection graph of the open intervals $(y,x)$ for $(x,y)\in V(G)$. Hence, in both cases $G$ is perfect, so $G$ contains either a clique or an independent set of size at least $n^{1/2}>\beta(1)n^{\gamma}$, which is also a cograph.

Now let $t\geq 2$, and let $\epsilon$ be the type of the vertices of $G$. As the single vertex graph is a cograph, this takes care of the base case $n=1$, so we can assume $n\geq 2$. For $i\in [t]$, there exist $p_i\in \mathbb{R}$ such that  at most $n/2$ of the vertices $(\mathbf{x},\mathbf{y})\in V(G)$ satisfy $\mathbf{x}(i),\mathbf{y}(i)\leq p_i$, and at most $n/2$ of the vertices satisfy $p_i \leq \mathbf{x}(i),\mathbf{y}(i)$. Let $V_i^{-}\subset V$ be the set of vertices $(\mathbf{x},\mathbf{y})$ satisfying $\mathbf{x}(i),\mathbf{y}(i)\leq p_i$,  let $V_i^{+}\subset V$ be the set of vertices satisfying $p_i\leq \mathbf{x}(i),\mathbf{y}(i)$, and let $V_{i}^{0}\subset V$ be the rest of the vertices.  Consider two cases. 

\textbf{Case 1.} There exists $i\in [t]$ such that $|V_{i}^{0}|\geq n/10t$. Note that if $\epsilon(i)=-$, then $G[V_i^{0}]$ is the empty graph, and if $\epsilon(i)=+$, then $G[V_i^{0}]$ is a quasi-comparability graph of complexity $t-1$ in which every vertex has the same type. Therefore, $G[V_i^0]$ contains a cograph of size at least $\beta(t-1)(n/10t)^\gamma=\beta(t) n^\gamma.$

\textbf{Case 2.} For every $i\in [t]$, we have $|V_{i}^{0}|<n/10t$. For simplicity, write $\beta$ instead of $\beta(t)$. For every $A\in 2^{[t]}$, let $$V_{A}=\bigcap_{i\in A}V_{i}^{+}\cap \bigcap_{i\in [t]\setminus A}V_{i}^{-}.$$
Note that $\{V_A\}_{A\subset [t]}$ forms a partition of  the set $U=V(G)\setminus (\bigcup_{i=1}^{t} V_{i}^{0})$, where 
$$|U|\geq n-\sum_{i=1}^{t}|V_i^{0}|\geq n-t\cdot \left(\frac{n}{10t}\right)\geq \frac{9n}{10}.$$ Also, note that if $A,B\in 2^{[t]}$ and $u\in V_{A}$ and $v\in V_{B}$, then there is no edge between $u$ and $v$ in $G$ if $A$ and $B$ are incomparable, and there is an edge between $u$ and $v$ if $\{A,B\}=\{\emptyset,[t]\}$. Consider the weight function $\omega:2^{[t]}\rightarrow\mathbb{R}_{\geq 0}$ defined as $\omega(A)=|V_A|/n$. Then $\omega(2^{[t]})\geq 9/10$, and $\omega$ is balanced as $\omega(F_{t}^{-}(i))=|V_{i}^{-}|/n$ and $\omega(F_{t}^{+}(i))=|V_{i}^{+}|/n$. Therefore, we can apply Lemma \ref{lemma:EH} to conclude that either (i) $\omega(\emptyset),\omega([t])\geq 1/10$, or (ii) there exist a positive integer $s$ and $s$ disjoint families $\mathcal{A}_1,\dots,\mathcal{A}_s\subset 2^{[t]}$ such that $$\sum_{i=1}^{s}\omega(\mathcal{A}_i)^{\gamma}\geq 1,$$
and for every $1\leq i<j\leq t$, every $A\in\mathcal{A}_i$ is incomparable to every $B\in \mathcal{A}_j$.
    
First, suppose that (i) holds. Then $V_{\emptyset}$ and $V_{[t]}$ both contain a cograph of size at least $\beta (n/10)^{\gamma}$. As every vertex in $V_{\emptyset}$ is connected to every vertex of $V_{[t]}$ by an edge, the union of these cographs is also a cograph. Therefore, $G$ contains a cograph of size at least $2\beta(n/10)^\gamma>\beta n^\gamma$.

Now consider the case if (ii) holds. For $i\in [s]$, the set $U_i=\bigcup_{A\in \mathcal{A}_i}V_{A}$ contains a cograph of size at least $\beta(\omega(\mathcal{A}_i)n)^{\gamma}$. Note that for $i\neq j$, there are no edges between $U_i$ and $U_j$, so the union of these cographs is a cograph as well, and its size is at least $$\sum_{i=1}^{s}\beta (\omega(\mathcal{A}_i)n)^{\gamma}\geq \beta n^\gamma.$$

We proved that every quasi-comparability graph of complexity $t$ on $n$ vertices in which the vertices have the same type contains a cograph of size at least $\beta n^{\gamma}$. But if $G$ is a quasi-comparability graph of complexity $t$ on $n$ vertices, then at least $n/2^t$ vertices have the same type. Also, every cograph on $m$ vertices contains a clique or an independent set of size at least $m^{1/2}$. Therefore, $G$ contains a clique or an independent set of size at least $\beta^{1/2}(n/2^t)^{\gamma/2}$. Hence, the choices $c=\gamma/2$ and $\alpha=\beta^{1/2}/2^{t\gamma/2}$ suffice.
\end{proof}

After these preparations, we get the proof of Theorem \ref{thm:sym} almost immediately.

\begin{proof}[Proof of Theorem \ref{thm:sym}]
Let $\alpha'=\alpha'(t)$ and $c$ be the constants given by Theorem \ref{thm:quasi}.  Similarly as in the proof of Theorem \ref{thm:mainthm}, if $G$ is semilinear of complexity $(t,u)$, then $G$ is the union of $u$ quasi-comparability graphs $G_1,\dots,G_u$ of complexity $t$. 

Therefore, $R_{t,1}(n)\leq (n/\alpha)^{1/c}$. Also, 
$$R_{t,u}(n)\leq R_{t,u-1}(R_{t,1}(n),R_{t,1}(n)).$$
Indeed, if $G$ has at least $R_{t,u-1}(R_{t,1}(n),R_{t,1}(n))$ vertices, then the graph $G_1\cup\dots\cup G_{u-1}$ contains either a clique or an independent set of size at least $R_{t,1}(n)$. A clique in this graph is also a clique in $G$. However, if it contains and independent set $I$ of size at least $R_{t,1}(n)$, then $G_u[I]$ contains either a clique or an independent set of size $n$, which is also a clique or an independent set of size $n$ in $G$.

But then by a simple induction argument, there exist $\alpha=\alpha(t,u)$ and $\beta=\beta(u)$ such that $$R_{t,u}(n)\leq \alpha n^{\beta}.$$
\end{proof}

\section{Constructions}\label{sect:construction}

In this section, we prove Theorems \ref{thm:mainlower} and \ref{thm:lower2}. Let us start with Theorem \ref{thm:mainlower}. We shall build on the following construction of Basit et al. \cite{BCSTT20}, which provides a $K_{2,2}$-free semilinear graph of constant complexity with superlinear number of edges.

\begin{lemma}\label{lemma:construction}
There exists $c_0>0$ such that the following holds. Let $n$ be a positive integer, then there exists a graph $G$ such that $G$ is the incidence graph of $n$ points and $n$ open rectangles in $\mathbb{R}^2$, $G$ is $K_{2,2}$-free, and $G$ has at least $c_{0}n\log n/\log\log n$ edges.
\end{lemma}

In order to get a triangle-free graph with small independence number, we use a technique of Alon and Pudl\'ak \cite{AP01}, which lets us transform a $K_{2,2}$-free graph with many edges into such a graph. Fortunately, this transformation preserves semilinearity as well.

If $G=(A,B;E)$ is a bipartite graph, and $\leq_{A},\leq_{B}$ are linear orders on $A$ and $B$, respectively, then $(G,\leq_A,\leq_B)$ is an \emph{ordered bipartite graph}. The \emph{super-line graph} of the ordered bipartite graph $(G,\leq_A,\leq_B)$ is the graph $H$, whose vertices are the edges of $G$, and for $u,u'\in A$ and $v,v'\in B$, if $\{u,v\},\{u',v'\}\in E(G)$, then $\{u,v\}$ and $\{u',v'\}$ are joined by an edge in $H$ if $u<_{A} u'$, $v<_{B} v'$, and $\{u,v'\}\in E(G)$. The following properties of the super-line graph can be found either in \cite{AP01} or \cite{ST21}.

\begin{lemma}\label{lemma:superline}
Let $(G,\leq_{A},\leq_{B})$ be an ordered bipartite graph such that both vertex classes of $G$ have size $n$, and let $H$ be the super-line graph of $(G,\leq_{A},\leq_{B})$. Then $H$ has no independent set of size larger than $2n$. Also, if $G$ has girth at least $2g$, then the girth of $H$ is at least $g+1$. In particular, if $G$ is $K_{2,2}$-free, then $H$ contains no triangles.
\end{lemma}

Now we are ready to describe our construction.

\begin{proof}[Proof of Theorem \ref{thm:mainlower}]
Let $n_0=n/2$, then by Lemma \ref{lemma:construction}, there exist a set $A\subset \mathbb{R}^{2}$  of size $n_0$, and a set $B$ of $n_0$ open rectangles in $\mathbb{R}^2$ such that the following holds. Let $G=(A,B;E)$ be the incidence graph between $A$ and $B$, then $G$ is $K_{2,2}$-free, and 
$|E|=N\geq c_{0}n_0\log n_0/\log\log n_0\geq c n\log n/\log\log n$, where  $c>0$ is chosen appropriately.

Let $\leq_{A},\leq_{B}$ be arbitrary linear orders on $A$ and $B$, respectively, and let $H$ be the super-line graph of $(G,\leq_{A},\leq_{B})$. Then $H$ has $N$ vertices, $H$ is triangle-free, and it contains no independent set of size $n$. It remains to show that $H$ is semilinear of  complexity $t$ for some constant $t$, as then $R_{t}(3,n)\geq N\geq cn\log/\log\log n$.

If $R\in B$ is a rectangle, then $R$ corresponds to the point $(a,b,c,d)\in \mathbb{R}^4$, where $(a,b)$ is the bottom left corner of $R$, and $(c,d)$ is the top right corner. A point $(x,y)$ is contained in $R$ if and only if $$\phi(x,y,a,b,c,d):=\{a-x<0\}\wedge \{b-y<0\}\wedge \{x-c<0\}\wedge\{y-d<0\}=\mbox{T}.$$  
Every vertex $w\in V(H)$ corresponds to a point $(x,y,a,b,c,d,i,j)\in \mathbb{R}^{8}$, where $(x,y)\in A$, $(a,b,c,d)\in B$, $i\in [n_0]$ is the position of $(x,y)$ with respect to the ordering $\leq_{A}$, and $j\in [n_0]$ is the position of $(a,b,c,d)$ with respect to the ordering $\leq_{B}$. But then, $(x,y,a,b,c,d,i,j),(x',y',a',b',c',d',i',j')\in V(H)$ are joined by an edge if and only if 
\begin{align*}
    &[\{i-i'<0\}\wedge \{j-j'<0\}\wedge \phi(x,y,a',b',c',d')]\vee\\
    &[\{i'-i<0\}\wedge\{j'-j<0\}\wedge \phi(x',y',a,b,c,d)]
\end{align*}
is true. Therefore, $H$ is semilinear of complexity 10.
\end{proof}

Next, we prove Theorem \ref{thm:lower2}. In particular, we show the extension of Lemma \ref{lemma:construction} that there are incidence graphs $G$ of $n$ points and $n$ rectangles, such that the girth of $G$ is at least $2g$, and $G$ has $\Omega_{g}(n\log\log n)$ edges. This then immediately implies Theorem \ref{thm:lower2} following the same proof as before. We omit the argument showing the implication, as it is straightforward. Let us remark that Davies \cite{D20} proved the existence of intersection graphs $G$ of $n$ boxes in $\mathbb{R}^3$, such that $G$ has girth at least $g$ and $|E(G)|\geq n\alpha_g(n)$, where $\alpha_g(n)$  tends to infinity extremely slowly as a function of $n$ compared to $\log\log n$. Note that an incidence graph of points and rectangles is also an intersection graph of boxes in $\mathbb{R}^3$: replace each point $(x,y)$ with a box $Q\times (0,1)$, where $Q$ is a small rectangle containing $(x,y)$, $Q$ intersects only those rectangles that contain $(x,y)$, and disjoint from all other rectangles $Q'$. Also, replace each rectangle $R$ with the box $R\times (x-\epsilon,x+\epsilon)$, where the intervals $(x-\epsilon,x+\epsilon)$ are pairwise disjoint and are contained in $(0,1)$.

\begin{theorem}\label{thm:girthedges}
Let $g$ be a positive integer, then there exists $c=c(g)>0$ such that the following holds.  Let $n$ be a positive integer, then there exists a graph $G$ such that $G$ is the incidence graph of $n$ points and $n$ open rectangles in $\mathbb{R}^2$, $G$ has girth at least $g$, and $G$ has at least $cn\log\log n$ edges.
\end{theorem}

Our proof is based on the argument in \cite{BCSTT20} of the proof of Lemma \ref{lemma:construction}, but we introduce a couple of new ideas as well. We prepare the proof with a few lemmas. Let $G=(A,B;E)$ be a bipartite graph and let $k$ be a positive integer. Define the bipartite graph $G\otimes k$ as follows: 
\begin{itemize}
    \item the vertex classes of $G$ are $A\times [k]$ and $(B\times [k])\cup A'$, where $A'$ is a copy of $A$,
    \item the vertices  $(u,i)\in A\times [k]$ and $(v,j)\in B\times [k]$ are joined by an edge if $i=j$ and $\{u,v\}\in E$,
    \item the vertices  $(u,i)\in A\times [k]$ and $v\in A'$ are joined by an edge if $u=v$.
\end{itemize}
Note that the vertex classes of $G\otimes k$ have sizes $|A|k$ and $|B|k+|A|$, and it has $|E|k+|A|k$ edges.

\begin{lemma}\label{lemma:realize}
Let $G$ be the incidence graph of points and open rectangles, and let $k$ be a positive integer. Then $G\otimes k$ can be realized as the incidence graph of points and open rectangles.
\end{lemma}

\begin{proof}
Let $X\subset \mathbb{R}^2$ be a set of points and $\mathcal{R}$ be a set of open rectangles whose incidence graph is $G$. After shifting and scaling, we may assume that each point and rectangle is contained in the open square $(0,1)^2$. Also, after slight perturbation, we can assume that no two points of $X$ share a coordinate. For $i\in [k]$, let $X_i$ and $\mathcal{R}_i$ be the copies of $X$ and $\mathcal{R}$ shifted by the vector $(0,i)$. Let $Y=X_1\cup\dots\cup X_k$. For each $x\in X$, we can find a thin rectangle $R_x$ that contains only the points $x_1,\dots,x_k$ from $Y$, where $x_i$ is the copy of $x$ in $X_i$. Let $\mathcal{Q}=\mathcal{R}_1\cup\dots\cup\mathcal{R}_k\cup\{R_x:x\in X\}$. Then the incidence graph of $Y$ and $\mathcal{Q}$ is isomorphic to $G\otimes k$.
\end{proof}

If $G$ is $K_{2,2}$-free, then $G\otimes k$ is $K_{2,2}$-free as well. In \cite{BCSTT20}, it is proved that if $k\approx \log n/\log \log n$, then the graph $((K_{1,k}\otimes k)\otimes\dots)\otimes k$ satisfies Lemma \ref{lemma:construction}, where $\otimes k$ is repeated $k$ times. We remark that this graph is $C_6$-free as well. However, if $k\geq 2$ and $G=(A,B;E)$ such that $B$ contains a vertex with degree at least $2$, then $G\otimes k$ contains cycles of length $8$. We show that we can sparsify $G\otimes k$ to avoid short cycles.

Let $H$ be a bipartite graph with vertex classes $A$ and $[k]$. Let $G\otimes_{H} k$ denote the induced subgraph of $G\otimes k$ in which we keep only those vertices $(u,i)\in A\times [k]$ for which $\{u,i\}\in E(H)$.

\begin{lemma}\label{lemma:girth}
Let $G=(A,B;E)$ be a bipartite graph with girth at least $2g$, and let $H$ be a bipartite graph with vertex classes $A$ and $[k]$. Suppose that $H$ contains no cycle $u_1,\ell_1,\dots,u_h,\ell_h$ with $u_1,\dots,u_h\in A$, $\ell_1,\dots \ell_h\in [k]$ for some $h<g$ such that $u_1,\dots,u_h$ have a common neighbor in $G$. Then $G\otimes_{H} k$ has girth at least $2g$.
\end{lemma}

\begin{proof}
Suppose that $w_1,\dots,w_{2p}$ is a cycle $C$ in $G\otimes_{H} k$ for some $p<g$, where $w_1\in A\times [k]$. Start walking on the vertices $w_1,\dots,w_{2p}$. If we are currently on $w_i$, and $w_i$ is an element of $A\times [k]$ or $B\times [k]$, then write down the first coordinate of $w_i$ and move to $w_{i+1}$. In case $w_i\in A'$, then skip $w_i$ and $w_{i+1}$, and move to $w_{i+2}$. Note that in this case the first coordinate of $w_{i-1}$ is the same as of $w_{i+1}$. Then we wrote down some sequence $u_1,v_1,\dots,u_{q},v_{q}$ for some $q\leq p$, where $u_i\in A$, $v_i\in B$ for $i\in [q]$, and $u_1,v_1,\dots,u_{q},v_{q}$ is a closed walk $W$ in $G$. Let $T$ be the graph consisting of the edges of this walk. As $G$ has girth at least $2g$, $T$ is a tree. Let $v=v_1$, and let $Z$ be the set of neighbors of $v$ in $T$. Put as many edges between $z,z'\in Z$ as many times $zvz'$ is a subwalk of $W$, and let $J$ be the resulting multigraph. Note that for each edge $\{z,z'\}\in E(J)$, there exists $\ell\in [k]$ such that $\{z,\ell\},\{z',\ell\}\in E(H)$.  Note that each edge $\{v,z\}$ is visited at least twice by $W$, so the minimum degree of $J$ is at least $2$. Therefore, $J$ contains a cycle $z_1,\dots,z_r$ for some $r\leq q$. But then there exists $\ell_1,\dots,\ell_r\in G$ such that $\{z_i,\ell_i\},\{z_i,\ell_{i+1}\}\in E(H)$ for $i\in [r]$, where indices are meant modulo $r$. Note that $\ell_i\neq \ell_j$ for $1\leq i<j\leq r$, because the unique neighbor of the vertices $(z_{i-1},\ell_i)$ and $(z_{i},\ell_i)$ in $G\otimes_{H} k$ is $(v,\ell_i)$, so $(v,\ell_i)$ would appear twice in the cycle $C$ otherwise. But then $z_1,\ell_1,\dots,z_r,\ell_r$ is a cycle of length $2r<2g$ in $H$, and $v$ is a common neighbor of $z_1,\dots,z_r$, which contradicts the condition on $H$.
\end{proof}

\begin{lemma}\label{lemma:findH}
Let $G=(A,B;E)$ be a bipartite graph such that every vertex in $B$ has degree at most $k$, $(100\log k)^{2g^2}|B|\leq |A|\leq k|B|$, and $G$ has girth at least $2g$. Then there exists a bipartite graph $H$ on vertex classes $A$ and $[k]$ such that $|E(H)|= \frac{1}{4}|A|^{1+1/2g}/|B|^{1/2g}$,  $G\otimes_{H} k$ has girth at least $2g$, and the maximum degree of every vertex in $(B\times [k])\cup A'$ is at most $2(|A|/|B|)^{1/2g}.$
\end{lemma}

\begin{proof}
Select each element of $A\times [k]$ with probability $p=\frac{2}{k}(\frac{|A|}{|B|})^{1/(2g)}$, and let $H_0$ be the resulting graph. Let $N$ be the number of edges of $H_0$, then $\mathbb{E}(N)= pk|A|$. Say that a cycle $u_1,\ell_1,\dots,u_h,\ell_h$ in $H_0$ is \emph{bad} if $u_1,\dots,u_h\in A$, $\ell_1,\dots \ell_h\in [k]$, and $u_1,\dots,u_h$ have a common neighbor in $G$.  For $h< g$, let $X_h$ be the number of bad cycles of length $2h$.  Then $$\mathbb{E}(X_h)\leq k^{2h}p^{2h}|B|,$$
as for every $v\in B$, there are at most $|N_{G}(v)|^{h}k^{h}\leq k^{2h}$ ways to choose $u_1,\ell_1,\dots,u_h,\ell_h$ such that $u_1,\dots,u_h\in N_{G}(v)$, and $\ell_1,\dots \ell_h\in [k]$, and the probability that the vertices $u_1,\ell_1,\dots,u_h,\ell_h$ form a cycle is $p^{2h}$. Also, let $X=\sum_{h=2}^{g-1}X_{h}$, then $$\mathbb{E}(X)\leq g|B|k^{2g-2}p^{2g-2}<\frac{pk}{4}|A|,$$ 
where the last inequality holds by the assumption $(100\log k)^{2g^2}|B|\leq |A|$.

Say that a vertex $w\in(B\times [k])\cup A'$ is \emph{bad} if its degree in $G\otimes_{H_0} k$ is more than $2pk$. Note that if $d$ is the degree of $w$, then $d$ is the sum of at most $k$ indicator random variables of probability $p$. Therefore, $\mathbb{E}(d)\leq pk$, and by the multiplicative form of Chernoff's bound, we have 
$$\mathbb{P}(w\emph{ is bad})=\mathbb{P}(d\geq 2pk)\leq e^{-\frac{pk}{3}}\leq \frac{1}{8k^2},$$
where the last inequality holds by the assumption $(100\log k)^{2g^2}|B|\leq |A|$.

Say that an edge $\{u,\ell\}\in E(H_{0})$ is bad if $(u,\ell)$ is a neighbor of a bad vertex in $G\otimes_{H_0} k$. Let $Y$ be the number of bad edges of $H_0$, then $$\mathbb{E}(Y)\leq k\cdot \frac{1}{8k^2}\cdot(|B|k+|A|)\leq \frac{|B|}{8}+\frac{|A|}{k}<\frac{pk}{4}|A|.$$
Let $H$ be the subgraph of $H_0$ we get after deleting every bad edge, and deleting one edge of every bad cycle of length less than $2g$. Let $M$ be the number of edges of $H$. Then
$$\mathbb{E}(M)\geq \mathbb{E}(N-X-Y)\geq \frac{pk}{2}|A|.$$
Hence, there is a choice for $H_0$ such that $H$ has at least $\frac{pk}{2}|A| =|A|^{1+1/2g}/|B|^{1/2g}$ edges. Note that $H$ contains no bad cycles of length less than $2g$, so $G\otimes_{H} k$ has girth at least $2g$ by Lemma \ref{lemma:girth}. Also every vertex in $(B\times [k])\cup A'$ has degree at most $2pk= 4(|A|/|B|)^{1/2g}$. Hence, $H$ satisfies the desired properties.
\end{proof}

Say that the tuple $(a,b,k,d,g)$ is \emph{realizable}, if there exists a bipartite graph $G=(A,B;E)$ such that $G$ is the incidence graph of points and open rectangles, $|A|=a$, $|B|=b$, every vertex in $B$ has degree at most $k$, every vertex in $A$ has degree exactly $d$, and $G$ has girth at least $2g$. In this case, say that $G$ \emph{realizes} $(a,b,k,d,g)$. Note that if $(a,b,k,d,g)$ is realizable for some $d\geq 1$, then $a\leq bk$. Indeed, if $G=(A,B;E)$ is a bipartite graph realizing $(a,b,k,d,g)$, then every vertex in $A$ has degree at least 1, and every vertex in $B$ has degree at most $k$, so $|A|\leq k|B|$.

\begin{lemma}\label{lemma:stepup}
Let $a,b,k,d,g$ be positive integers such that $(a,b,k,d,g)$ is realizable, and $(100\log k)^{2g^2}b\leq a$. Then 
$$\left(\frac{a^{1+1/2g}}{b^{1/2g}},2kb,
4\left(\frac{a}{b}\right)^{1/2g},d+1,g\right)$$
is also realizable.
\end{lemma}

\begin{proof}
Let $G=(A,B;E)$ be a graph realizing $(a,b,k,d,g)$. Let $H$ be a bipartite graph with vertex classes $A$ and $[k]$, and let $G_0=G\otimes_{H} k$ with vertex classes $A_0$ and $B_0$.  Then $G_0$ is also an incidence graph of points and rectangles by Lemma \ref{lemma:realize}, $|A_0|=|E(H)|$, $|B_0|=a+kb\leq 2kb$, and every vertex in $A_0$ has degree $d+1$. Choose $H$ satisfying the conditions of  Lemma \ref{lemma:findH}, then $G_0$ has girth at least $2g$, $|A_0|=\frac{a^{1+1/2g}}{b^{1/2g}}$, and every vertex in $B_0$ has degree at most $4\left(\frac{a}{b}\right)^{1/2g}$. Add some isolated vertices to $B_0$ in order to make its size equal to $2bk$. Then, $G_0$ realizes $$\left(\frac{a^{1+1/2g}}{b^{1/2g}},2kb,
4\left(\frac{a}{b}\right)^{1/2g},d+1,g\right).$$

\end{proof}

\begin{proof}[Proof of Theorem \ref{thm:lower2}]
Clearly, the tuple $(a_1,b_1,k_1,d_1,g)=(m,1,m,1,g)$ is realizable, as the incidence graph of a  single rectangle containing $m$ points realizes it. For $i=1,2,\dots$, as long as $(100\log k_i)^{2g^2}b_i\leq a_i$ holds, define 
$$(a_{i+1},b_{i+1},k_{i+1},d_{i+1},g):=\left(\frac{a_i^{1+1/2g}}{b_i^{1/2g}},2k_i b_i,4\left(\frac{a_i}{b_i}\right)^{1/2g},d_i+1,g\right),$$
and let $I$ be the last index for which this tuple is defined. Then by Lemma \ref{lemma:stepup}, $(a_i,b_i,k_i,d_i,g)$ is realizable for $i\leq I$. Note that $d_i=i$, and $a_{i}\geq b_i$ for $i\leq I-1$. Let  $G=(A,B;E)$ be the bipartite graph realizing $(a_{I-1},b_{I-1},k_{I-1},d_{I-1},g)$. Remove some vertices from $A$ to make $A$ and $B$ have equal size $n$. Then $G$ is the incidence graph of points and rectangles, $G$ has girth at least $2g$, and every vertex in $A$ has degree equal to $I-1$, which implies that $G$ has $(I-1)n$ edges. Therefore, to finish the proof, it is enough to show that $I=\Omega_g(\log\log m)$ and $a_{I-1}\leq a_{I}=O_{g}(m^{2})$.

Let $\gamma_0=0$, and for $i=1,\dots,I$, let $$\gamma_{i}=\sum_{j=0}^{i-1}\left(\frac{1}{2g}\right)^{j}.$$ Clearly, $\gamma_i<2$. It is not difficult to calculate that $a_i=\alpha_i m^{\gamma_i},$ $b_i=\beta_i m^{\gamma_{i-1}}$ and $k_i=\kappa_i m^{1/(2g)^{i-1}}$, where $\alpha_i^{1/i},\beta_{i}^{1/i},\kappa_i^{1/i}=\Theta_g(1)$. Therefore, $(100\log k_I)^{2g^2}b_{I}> a_{I}$ implies $$(100\log m/(2g)^{I-1})^{2g^2}>\Theta_g(1)^{I}m^{1/(2g)^{I-1}}.$$ 
But then $I= \Theta_g(\log\log m)$ and $a_{I}=O_g(m^2)$.
\end{proof}

\section{Concluding remarks}

In this paper, we proved that $R_{t}(s,n)=O_{t}(1)\cdot s^{O_{t}(1)}\cdot n(\log n)^{O_t(1)}.$ Also, we showed that the logarithmic term cannot be omitted. 

\begin{itemize}
    \item Can we replace $(\log n)^{O_t(1)}$ with simply $\log n$? We suspect that the answer is no, and it would be interesting to see examples showing that a larger power of $\log n$ is indeed needed.
    \item What happens if we assume large girth? Are there semilinear graphs of bounded complexity with girth $g$, independence number $n$, and at least $n(\log n)^{\Omega_g(1)}$ vertices?
\end{itemize}

What can we say about the Ramsey properties of semilinear $r$-uniform hypergraphs for ${r\geq 3}$? (We defined semilinearity only for graphs, but the definition extends for hypergraphs in a straightforward way.)  It was proved by Conlon et al. \cite{CFPSS14} that if $\mathcal{H}$ is an $r$-uniform semialgebraic hypergraph on $N$ vertices of complexity $t$ which contains no clique or independent set of size $n$, then $N\leq \mbox{tw}_{r-1}(n^C)$ for some constant $C=C(t)$, where the \emph{tower function} $\mbox{tw}_k(x)$ is defined as $\mbox{tw}_1(x):=x$ and $\mbox{tw}_{k+1}(x):=2^{\mbox{tw}_{k}(x)}$. This bound is also the best possible up to the value of~$C$. Do the Ramsey numbers of $r$-uniform semilinear hypergraphs behave similarly?
\begin{itemize}
    \item Let $R^{r}_t(s,n)$ denote the smallest $N$ such that every $r$-uniform semilinear hypergraph of complexity $t$ on  $N$ vertices contains either a clique of size $s$, or an independent set of size $n$. What is the order of $R^{r}_t(s,n)$ when $s$ is a constant? Also, what is the order of $R^{r}_t(n,n)$?
\end{itemize}

\vspace{0.3cm}
\noindent	
{\bf Acknowledgements.} The author acknowledges the support of the SNSF grant 200021\_196965, the support of the Russian Government in the framework of MegaGrant no 075-15-2019-1926, and the support of MIPT Moscow.

\end{document}